%% file: CupProducts_FinalAccepted_ForumMathematicum.tex
\documentclass[a4paper]{amsart}
\usepackage{amssymb}
\usepackage[utf8]{inputenc}
\usepackage{microtype} 
\usepackage{amsmath,amsfonts,amscd,amsthm,bbm,enumerate}
\usepackage{graphicx,xcolor} 
\usepackage{epstopdf}
\usepackage{mathtools} 
\usepackage{mathrsfs} 
\usepackage{multirow}
\usepackage{tikz}
\usepackage{tikz-cd}
\usetikzlibrary{matrix,arrows}
\usepackage{transparent}
\usepackage[hidelinks]{hyperref}
\usepackage{cancel} 

\theoremstyle{plain}
 \newtheorem{thm}{Theorem}
 \newtheorem{prop}[thm]{Proposition}

 \newtheorem*{op*}{Open Problem} 
 \newtheorem*{frthmA*}{Theorem A} 
 \newtheorem*{frcorB*}{Corollary B} 
  \newtheorem*{frthmAp*}{Theorem A'} 
   
\theoremstyle{definition}
 \newtheorem{exm}[thm]{Example}
 \newtheorem{dfn}[thm]{Definition}
\theoremstyle{remark}
 \newtheorem*{rem}{Remark}

\newcommand{\R}{\mathbb{R}} 

\newcommand{\Z}{\mathbb{Z}} 
\newcommand{\p}{\mathcal{P}} 
\newcommand{\s}{\mathcal{S}} 

\DeclareMathOperator{\Ker}{Ker}
\DeclareMathOperator{\Img}{Im}
\DeclareMathOperator{\id}{\mathrm{id}}
\DeclareMathOperator{\h}{\mathrm{H}} 
\DeclareMathOperator{\C}{\mathrm{C}} 
\DeclareMathOperator{\QM}{\mathrm{QM}}

\renewcommand{\leq}{\leqslant}
\renewcommand{\geq}{\geqslant}
\renewcommand{\setminus}{\smallsetminus}
\title[CUP PRODUCTS IN BOUNDED COHOMOLOGY VIA ALIGNED CHAINS]{TRIVIAL CUP PRODUCTS IN BOUNDED COHOMOLOGY OF THE FREE GROUP VIA ALIGNED CHAINS}

\author[Sofia Amontova]{
Sofia Amontova} 

\address{ 
Universit\'e de Gen\`eve
}
\email{Sofia.Amontova@unige.ch}

\author[Michelle Bucher]{
Michelle Bucher} 
\address{Universit\'e de Gen\`eve}
\email{Michelle.Bucher-Karlsson@unige.ch}

\thanks{Supported by the Swiss National Science Foundation, NCCR SwissMAP} 

\begin{document}

\begin{abstract}
We prove that the cup product of $\Delta$-decomposable quasimorphisms,  Brooks quasimorphisms or Rolli quasimorphisms with any bounded cohomology class of arbitrary positive degree is trivial. \end{abstract}

\maketitle

\section{Introduction}

Despite its wide range of uses in mathematics, bounded cohomology turns out to be hard to compute in general. 
Even for the case of a non-abelian free group $F$, while something can be said about the bounded cohomology group with trivial real coefficients $\h_b^n(F, \R)$ up to degree $3$, for degrees $4$ and higher it is still not known whether $\h_b^n(F, \R)$ vanishes or not.
Expecting the former to hold, we consider  the following weaker question:
\begin{op*} For all $k >0$ is the cup product 
	\[
		\cup: \h_b^2(F,\R) \times \h_b^k(F,\R) \longrightarrow \h_b^{k+2}(F,\R) 
	\]
trivial?
\end{op*}
In recent years a whole variety of examples of trivial cup products have been exhibited in the case $k=2$, namely
it holds that:
\begin{thm}\label{summary}
For all quasimorphisms $\phi$ and $\psi$ satisfying either of the following conditions
\begin{enumerate}
	\item[(a)] \emph{(Bucher, Monod, \cite{buchermonod18})} $\phi$ and $\psi$ are Brooks quasimorphisms,
	\item[(b)] \emph{(Heuer, \cite{heuer17})} $\phi$ is a $\Delta$-decomposable quasimorphism and $\psi$ a $\Delta$-con\-tin\-u\-ous quasimorphism,
	\item[(c)] \emph{(Heuer, \cite{heuer17})} Each of $\phi$ and $\psi$ is either a Brooks quasimorphism on non-selfoverlapping words  or a Rolli quasimorphism,
	\item[(d)] 
	\emph{(Fournier-Facio, \cite{facio20})}  $\phi$ and $\psi$ are certain Calegari quasimorphisms,\footnote{We refer to \cite{facio20} for the precise statement and the definition of Calegari quasimorphisms.}
\end{enumerate}
the cup product
\begin{align*}
\cup: \h_b^2(F,\R) \times \h_b^2(F,\R) &\longrightarrow \h_b^4(F,\R)\\
	([\delta^1\phi], [\delta^1\psi]) &\longmapsto [0]
\end{align*}
vanishes. 
\end{thm}

Although the publication dates of Theorem \ref{summary} (a) and (b)-(c) differ, these two results, which have a significant overlap, were announced at the same conference in 2017. 

In the present paper we provide an intermediate result towards answering the Open Problem for all $k > 0$:

\begin{frthmA*}
Let $k>0$ and $\alpha\in \h_b^k(F,\R)$ be arbitrary. Let $\phi$ be a quasimorphism. Then the cup product 
\begin{align*}
\cup: \h_b^2(F,\R) \times \h_b^k(F,\R) &\longrightarrow \h_b^{k+2}(F,\R)\\
	([\delta^1\phi], \alpha) &\longmapsto [0]
\end{align*}
vanishes whenever 
\begin{enumerate}
	\item[(a)] $\phi$ is a $\Delta$-decomposable quasimorphism,
	\item[(b)] $\phi$ is a Brooks quasimorphism,
	\item[(c)]  $\phi$ is a Rolli quasimorphism. 
\end{enumerate}
\end{frthmA*}

Even in the case of $k=2$, our Theorem A is a priori stronger than Theorem \ref{summary}, since it is not known whether every quasimorphism is $\Delta$-decomposable, $\Delta$-continuous or a Calegari quasimorphism. Moreover for $k=3$, since both $\h_b^2(F,\R)$ and $\h_b^3(F,\R)$ are known to be infinite dimensional (see \cite{brooks81}, \cite{Grigorchuk} as well as \cite{soma97}),  our Theorem gives infinitely many examples of the triviality of the cup product of non trivial classes.

Important examples of $\Delta$-decomposable quasimorphisms are the classical Brooks quasimorphisms on non-selfoverlapping words and Rolli quasimorphisms. Thus, Theorem A (c) is a particular case of Theorem A (a). For Brooks quasimorphisms on selfoverlapping words we are however not aware of a construction as $\Delta$-decomposable quasimorphisms. As a consequence, we will deal with case (b) separately in Section \ref{Section Brooks}.  



Similar to the generalisation of Heuer's vanishing result in Theorem \ref{summary}(b) by Fournier-Facio in Theorem \ref{summary}(d), we expect the generalization of Theorem A to Fournier-Facio's setting  to be straightforward following \cite{facio20}.
\\ \\
Our strategy to prove Theorem A (a) is to apply the geometric proof from \cite{heuer17} to the setting of aligned cochains from \cite{buchermonod19}. On the one hand, Nicolaus Heuer introduced in \cite{heuer17} the notion of $\Delta$-decomposition which abstracts the geometric properties of both Brooks quasimorphisms on non-selfoverlapping words and Rolli quasimorphisms needed to prove Theorem \ref{summary}(b). On the other hand, the aligned chain complex is a geometric construction introduced by Monod and the second author in \cite{buchermonod19} to prove the vanishing of the continuous bounded cohomology of automorphism groups of trees and also the vanishing of the cup product of Brooks quasimorphisms (Theorem \ref{summary}(a)) in \cite{buchermonod18}. In our setting, the aligned cochain complex allows for consequent simplifications in the proof by Heuer to the point where we can relax the hypothesis on the second factor of the cup product in Theorem \ref{summary}(b) significantly; that is, from being a $\Delta$-continuous quasimorphism, and hence an a priori strict subset of $\h^2_b(F)$, to an arbitrary cohomology class in any degree $k>0$.

The proof of Theorem A (b) is essentially a translation of the combinatorial aspects of the proof of Theorem A (a) to Brooks quasimorphisms and does not rely on $\Delta$-decompositions. For a unified approach to these two proofs, see the last remark of Section \ref{section: proof}. 


The paper is structured as follows: In Section \ref{section: bounded cohomology} we recall a selection of basic notions on bounded cohomology theory. Section \ref{section: aligned cochain complex} introduces the inhomogeneous aligned cochain complex, where we follow \cite{buchermonod19}. The proof of Theorem A (b) for Brooks quasimorphisms is given in Section \ref{Section Brooks}. Then to define 
$\Delta$-decompositions, the induced $\Delta$-decomposable quasimorphisms and their special cases of Brooks quasimorphisms and Rolli quasimorphisms, we follow \cite{heuer17} in Section \ref{section: decomposition}. Finally, Theorem A (a) is proven in Section \ref{section: proof}.  

\subsection*{Acknowledgements} This paper is part of a master thesis project by the first author under the supervision of the second author within the framework of the Master Class program organized by NCCR SwissMAP. We are especially grateful to Nicolas Monod for his interesting feedback while refereeing the thesis and to Pierre de la Harpe, Francesco Fournier-Facio, Clara L{\"o}h and Marco Moraschini for several{} useful comments. We also wish to thank an anonymous referee for his or her pertinent corrections. 

\section{Bounded cohomology}\label{section: bounded cohomology}

We shortly define bounded cohomology, originally introduced in \cite{Gromov}. For more details, the reader is referred to \cite{frigerio17}.

\subsection{Inhomogeneous cochains and bounded cohomology} We give the definition of (bounded) cohomology of a group $G$ with trivial real coefficients by using the inhomogeneous  resolution.
For all $n \geq 0$ we denote by
\begin{align*}
\C^n(G, \R) &= \{\varphi: G^n \to \R\} \text{ and }\\
\C^n_{b}(G, \R) &= \{\varphi\in C^n(G,\R) \mid  ||\varphi||_\infty < \infty\} ,
\end{align*}
where
\[
||\varphi ||_\infty := \sup\left\{|\varphi(g_1,\dots,g_n)| \mid (g_1, \dots, g_n) \in  G^n\right\} \in [0,+\infty],
\]
the set of \emph{cochains} and \emph{bounded} cochains respectively. 
The \emph{coboundary map} $\delta^n: \C^n(G, \R) \to \C^{n+1}(G, \R)$ is defined by
\begin{align*}
\delta^n(\varphi)(g_1, \dots, g_{n+1})=
&\varphi(g_2, \dots, g_{n+1}) \\
&+\sum_{j=1}^n(-1)^j \varphi(g_1, \dots, g_jg_{j+1}, \dots, g_{n+1})\\
&+(-1)^{n+1} \varphi(g_1, \dots, g_n)
\end{align*}
and it restricts to the \emph{{bounded} coboundary map} $\delta^n: \C^n_{b}(G, \R) \to \C^{n+1}_{b}(G, \R)$.
Since $\delta^{n+1} \circ \delta^{n} = 0$ holds for all $n \geq 0$, we have that 
$\left(\C^*_{(b)}(G,\R), \delta^*\right)$
is a \emph{(bounded)} cochain complex. 
Moreover, the {(bounded)} $n$-cochain $\varphi$ is called a
\emph{(bounded) $n$-cocycle} if $\delta^{n+1}\varphi=0$, and a
 \emph{(bounded) $n$-coboundary} if there exists a \textit{(bounded) $(n-1)$-cochain} $\psi$ such that $\delta^{n-1}\psi=\varphi$.
The {(bounded)} cohomology of $G$ is then defined as the cohomology of the {(bounded)} cochain complex. More precisely:
\begin{dfn}
 The $n$-th \emph{{(bounded)} cohomology of $G$ with trivial coefficients} is defined as
 	\[
 		\h^n_{(b)}(G,\R) :
 		= \frac{\Ker\left[\delta : \C^n_{(b)}(G,\R) \to \C^{n+1}_{(b)}(G,\R)\right]}
 		{\Img\left[\delta : \C^{n-1}_{(b)}(G,\R) \to \C^n_{(b)}(G,\R)\right]}.
 	\]
\end{dfn}
We shall denote by $\delta^*$ both the bounded and the unbounded coboundary maps. 

\subsection{Cup product}
We introduce the operation that allows to obtain elements of {(bounded)} cohomology groups of higher degrees from elements of {(bounded)} cohomology groups of lower degrees.
\begin{dfn}
The \textit{cup product} is the map given by
\begin{align*}
\cup: \h^n(G, \R) \times \h^m(G, \R) &\longrightarrow \h^{n+m}(G, \R)\\
([f], [g]) &\longmapsto [f] \cup [g],
\end{align*}
where $[f] \cup [g]$ is represented by the cocycle $f \cup g \in \C^{n+m}(G,\R)$ that is defined as follows
\[
f \cup g: (g_1, \dots, g_n, g_{n+1}, \dots, g_{n+m}) \longmapsto f(g_1, \dots, g_n)\cdot g(g_{n+1}, \dots, g_{n+m}).
\]
This map induces a well defined map on {bounded} cohomology that we also call the cup product
	\[
	\cup: \h^n_{b}(G, \R) \times \h^m_{b}(G, \R) \longrightarrow \h^{n+m}_{b}(G, \R).
	\]
\end{dfn}

\subsection{Quasimorphisms on a free group} 
We recall that for any group $G$ a \emph{quasimorphism} is a map $\phi: G \to \R$ for which there exists a constant $D>0$ such that
\[
\sup_{g,h \in G} |\phi(g)+\phi(h)-\phi(gh)| < D.
\]
There is a straightforward connection to bounded cohomology: for any such $\phi$ 
the norm of $||\delta^1\phi||_\infty$ is bounded by $D$ and therefore $\delta^1 \phi \in \C_b^2(G, \R)$ is a cocycle and represents a cohomology class in $\h^2_b(G,\R)$. 

For the case of a non-abelian free group $F$, the study of $\h_b^2(F,\R)$ boils down to studying the set of quasimorphisms on $F$ which we denote by $\QM(F)$, since the following holds
	\[
	\h_b^2(F,\R)  \cong \frac{\QM(F)}{[\mathrm{Hom}(F,\R) \oplus \ell^\infty(F)]},
	\]
where $\mathrm{Hom}(G,\R)$ and $\ell^\infty(G)$ denote the set of homomorphisms and real-valued bounded maps on $G$ respectively (see for example \cite[Proposition 3.2]{Grigorchuk}).
For instance, with help of various explicit constructions of certain families of quasimorphisms it was proved that $\h_b^2(F, \R)$ is non-vanishing and moreover infinite-dimensional, e.g. \textit{Brooks quasimorphisms}, due to Brooks in \cite{brooks81} (see also \cite{Grigorchuk}), and \textit{Rolli quasimorphisms}, due to Rolli in \cite{rolli09}.

\begin{exm}{\emph{Brooks quasimorphisms.}}\label{Brooks quasimorphisms}
Let $\mathcal{S} $ be the symmetrization of a free generating set of $F$ and $w$ be any reduced word in $F$. 
Then a \emph{Brooks quasimorphism} on the word $w$ is a function $\phi : F \to \Z$ loosely defined as
	\[
	\phi_w(g) = \#\{\text{subwords $w$ in g}\} - \#\{\text{subwords $w^{-1}$ in g}\},\quad \text{for all $g \in F$}.
	\]
	(For a precise definition, we refer to Section \ref{Section Brooks}.) 
A reduced word $w$ in the non-abelian free group $F$ is called  a \emph{non-selfoverlapping word}  if there do not exist words $s$ and $m$ with $s$ non-trivial such that $w=sms$ as a reduced word.  In this case $\phi_w$ counts disjoint occurences of $w^{\pm1}$. Otherwise we say that $w$ is a \emph{selfoverlapping word} and  then $\phi_w$ counts occurences of $w$ which may overlap. 


\end{exm}

\begin{exm}{\emph{Rolli quasimorphisms}.}
Let $\s = \{x_1, \dots, x_\mathbf{n}\}\cup\{x_1^{-1},\dots,x_\mathbf{n}^{-1}\} $ be the symmetrization of a free generating set of $F$. 
A \emph{Rolli quasimorphism} is a map $\phi: F \to \Z$ defined as
\[
g \mapsto \sum_{i=1}^k \lambda_{j_i}(m_i), 
\]
where
\begin{itemize} 
\item $m_1, \dots, m_k$ are the unique non-zero integers such that it holds that
\[
g = x_{j_1}^{m_1} \cdot \ldots \cdot x_{j_k}^{m_k},
\]
where no consecutive $j_i$ are the same.
\item $\lambda_1, \dots, \lambda_\mathbf{n}$ are bounded alternating\footnote{$\lambda: \Z \to \R$ is \emph{alternating}, if it holds that $\lambda(-n)=-\lambda(n)$ for all $n \in \Z$.} functions.
\end{itemize}
\end{exm}
We encounter another viewpoint for these examples of quasimorphisms in Section \ref{section: decomposition}.

\section{The inhomogeneous aligned cochain complex}\label{section: aligned cochain complex}

In this section we briefly present the inhomogeneous version of aligned cochain complexes of non-abelian free groups first introduced in \cite{buchermonod19} in the homogeneous context to prove the vanishing of continuous bounded cohomology for certain automorphisms groups of a tree. 

Let again $F=F_\mathbf{n}$ be a free group of rank $\mathbf{n}$ and $\s$ be the symmetrization of a generating set of $F$. 

For any $n \geq 0$ we define
	\[
		\mathcal{B}^{n} = \left\{(g_1,\dots,g_n)\in F^n \ \bigg | \
		\parbox[c]{2,9in}{ $g_ig_{i+1}$ is reduced with respect to $\s$ \;\;  $\forall 1\leq i < n$,\\
		 $g_i\neq \mathrm{id} \;\; \forall 1\leq i \leq n$.}
		\right\}.
	\]

Geometrically, this means that in the Cayley graph of $G$ with respect to $\s$ the corresponding homogenous $(n+1)$-tuple $(\mathrm{id},g_1,g_1g_2,\dots,g_1\cdot \ldots \cdot g_n)$ of distinct vertices is ordered on the geodesic segment between $\mathrm{id}$ and $g_1\cdot \ldots \cdot g_n$ (see Figure \ref{triangle_picture2}).

	\begin{figure}[h]
			\center
			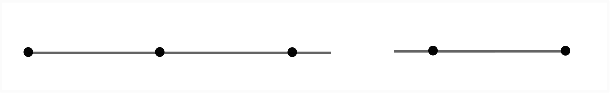
			\caption{}
			\label{triangle_picture2}
		\end{figure}

\begin{dfn} The \emph{(inhomogeneous) aligned cochain complex}  is defined as the set of functions
$$\mathscr{A}^n(F, \R) =\{\varphi: \mathcal{B}^n \longrightarrow \R\}.$$
The \emph{bounded (inhomogeneous) aligned chain complex} $\mathscr{A}_b^n(F, \R) $ consists of bounded aligned cochains:
$$\mathscr{A}^n_b(F, \R) =\{ \varphi\in \mathscr{A}^n(F, \R)\mid \| \varphi\|_\infty<\infty\}. $$
\end{dfn}

\begin{dfn} The \emph{(bounded) (inhomogeneous) aligned alternating  cochain complex} $\mathscr{A}^n_{+,(b)}(F, \R)$ is defined as the set of functions $\varphi\in \mathscr{A}^n_{(b)}(F, \R)$ such that 

$$\varphi(g_1,\dots,g_n)=(-1)^{\lceil n/2 \rceil} \varphi(g_n^{-1},\dots, g_1^{-1})$$
for every $(g_1,\dots,g_n)\in  \mathcal{B}^{n}$. 
\end{dfn}

These four complexes are endowed with their inhomogenous coboundary 
\[
\boldsymbol{\delta}^n: \mathscr{A}^n_{(+),(b)}(F, \R) \longrightarrow \mathscr{A}^n_{(+),(b)}(F, \R).\] 
There is a chain map 
$$A: \mathscr{A}^n_{(b)}(F, \R)\longrightarrow \mathscr{A}^n_{+,(b)}(F, \R)$$
defined as
$$A(\varphi)(g_1,\dots,g_n)=\frac{1}{2}\left(\varphi(g_1,\dots,g_n)+(-1)^{\lceil n/2 \rceil}\varphi(g^{-1}_n,\dots,g^{-1}_1)\right),$$
for every $\varphi \in \mathscr{A}^n_{(b)}(F, \R)$ and $(g_1,\dots,g_n)\in  \mathcal{B}^{n}$. 

Since $\mathcal{B}^n$ is a subset of $F^n$, we have by restriction natural maps
$$r: \C^n_{(b)}(F,\R)\longrightarrow \mathscr{A}^n_{(b)}(F, \R)$$
which commute with the coboundary maps.  It is proven in \cite[Proposition 8]{buchermonod18} that the composition of the chain maps 
$$A\circ r: \C^n_{b}(F,\R)\longrightarrow \mathscr{A}^n_{+,b}(F, \R)$$
induces an isomorphism between the (bounded) cohomology of the free group $F$ and the cohomology of the cocomplex $\left(\mathscr{A}^*_{(b)}(F, \R), \boldsymbol{\delta}^*\right)$. 

Let now $\alpha_1, \alpha_2 \in H^*_b(F,\R)$ be represented by inhomogenous cocoycles $\omega_1\in \C^n_{b}(F,\R)$ and $\omega_2\in \C^m_{b}(F,\R)$. To show that the cup product $\alpha_1\cup \alpha_2$ vanishes it is hence sufficient to show that 
$$A\circ r(\omega_1\cup\omega_2)=A(r(\omega_1)\cup r(\omega_2))$$
is a coboundary in  $\mathscr{A}^{n_1+n_2}_{+,b}(F, \R)$. This immediately follows if $r(\omega_1)\cup r(\omega_2)$ is a coboundary in $\mathscr{A}^{n_1+n_2}_{b}(F, \R)$. Indeed, if 
$$r(\omega_1)\cup r(\omega_2)=\boldsymbol{\delta}\varphi,$$
for some $\varphi\in \mathscr{A}^{n_1+n_2-1}_{b}(F, \R)$ then
$$A(r(\omega_1)\cup r(\omega_2))=A\boldsymbol{\delta}\varphi=\boldsymbol{\delta}(A\varphi).$$
This is the strategy we will use to prove Theorem A. 

\section{Proof of Theorem A (\normalfont{\lowercase{b}})}\label{Section Brooks}

Let $\mathcal{S} $ be the symmetrization of a free generating set of $F$ and $w$ be any nontrivial element of $F$, i.e. a nonempty reduced word in $\mathcal{S} \cup \mathcal{S} ^{-1}$. Define a function $\chi_w:F\rightarrow \{ +1,0,-1\}$ as
$$\chi_w(g)=
\left\{ \begin{array}{rl}
1, &\mathrm{if \ } g=w,\\
-1, &\mathrm{if \ } g=w^{-1},\\
0, &\mathrm{else}.
\end{array}\right.
$$
Let $\ell$ denote the length of $w$. The Brooks quasimorphism on $w$ is the function $\phi_w:F\rightarrow \mathbb{Z}$  defined by
$$\phi_w(g)=\sum_{j=1}^{m-\ell+1}\chi_w(x_j \cdot \ldots \cdot x_{j+\ell-1}),$$
if $g=x_1\cdot \ldots \cdot x_m$ is a reduced expression of length $m$, where the sum is considered to be equal to $0$ if the index set is empty, i.e. if $m<\ell$.

\begin{proof}[Proof of Theorem A (b)]

Due to the above discussion we can prove Theorem A (b) by showing that the cup product of the restriction to aligned cochains of cocycles representing the cohomology classes in consideration is a coboundary. Thus, it is enough to consider the restriction of Brooks quasimorphisms to aligned chains, which we still denote by $\phi_w \in \mathscr{A}^1(F,\R)$, and to show that for any $k >0$ and $\omega \in \mathscr{A}_b^k(F,\R)$, there exists $\beta  \in \mathscr{A}_b^{k+1}(F,\R)$ such that 
$$\boldsymbol{\delta} \beta = \boldsymbol{\delta}\phi_w \cup \omega.$$
Let $\eta \in \mathscr{A}^k(F, \R)$ be defined by 
\begin{equation}\label{eq: choiceetaBROOKS}
\eta: (g,h_1,\dots,h_{k-1}) \longmapsto \sum_{j=1}^{m-\ell+1} \chi_w( x_j \cdot \ldots \cdot x_{j+\ell-1})\omega(z_{j+\ell}(g),h_1,\dots,h_{k-1}),
\end{equation}
where $g=x_1\cdot \ldots \cdot x_m$ is a reduced expression of length $m$, and $z_j(g)=x_{j}\cdot \ldots \cdot x_m$. 

Define  $\beta  \in \mathscr{A}_b^{k+1}(F,\R)$ by
$$ \beta=\phi_w \cup \omega +\boldsymbol{\delta} \eta.$$
By construction, $\boldsymbol{\delta} \beta = (\boldsymbol{\delta}\phi_w) \cup \omega$, so that we only need to show that $\beta$ is bounded. Let $(g,h_1,\dots,h_k)$ in $\mathcal{B}^{n}$ be arbitrary. Since this chain is aligned, the concatenation of a reduced expression for $g$ and $h_1$ respectively, is a reduced expression (obviously for $gh_1$). Let $g=x_1\cdot \ldots \cdot x_m$ and $h_1=x_{m+1}\cdot \ldots \cdot x_{m+n}$ be reduced expressions. 
 
Then by definition, we have that
\begin{align}\label{eq01BROOKS}
\beta(g,h_1,\dots,h_k)
=&\phi_w(g)\omega(h_1,\dots,h_k) \\
+&\Big\{\eta(h_1,\dots,h_k)
-\eta(gh_1, h_2\dots,h_k)\nonumber\\
-&\sum_{i=1}^{k-1}(-1)^i\eta(g,h_1,\dots,h_ih_{i+1},\dots,h_k)\nonumber\\
-&(-1)^k\eta(g,h_1,\dots,h_{k-1})\Big\}\nonumber.
\end{align}
By expanding $\eta$ as in its definition \eqref{eq: choiceetaBROOKS} we obtain 
\begin{align}\label{eq02BROOKS}
&\quad\beta(g,h_1,\dots,h_k) \\[0.3em]&= 
\phi_w(g)
\omega(h_1,\dots,h_k)\nonumber\\
&+\sum_{j=m+1}^{m+n-\ell+1}\chi_w(x_j \cdot \ldots \cdot x_{j+\ell-1})\omega(z_{j+\ell}(h_1), h_2, \dots, h_k)\nonumber\\
&-\sum_{j=1}^{m+n-\ell+1}\chi_w(x_j \cdot \ldots \cdot x_{j+\ell-1})\omega(z_{j+\ell}(gh_1), h_2, \dots, h_k)\nonumber\\
&-\sum_{i=1}^{k-1}(-1)^i \sum_{j=1}^{m-\ell-1} \chi_w(x_j \cdot \ldots \cdot x_{j+\ell-1})\omega(z_j(g), h_1,\dots,h_ih_{i+1},\dots,h_k)\nonumber\\
&-(-1)^k \sum_{j=1}^{m-\ell+1}\chi_w(x_j \cdot \ldots \cdot x_{j+\ell-1})\omega(z_j(g),h_1,\dots,h_{k-1}).\nonumber
\end{align}
Using the definition of $\phi_w(g)$ and the cocycle relation 
\begin{align} \label{eq03BROOKS}
\omega(z_j(g)h_1,\dots,h_k)
=&\omega(h_1,\dots,h_k)\\
-&\sum_{i=1}^{k-1}(-1)^i\omega(z_j(g),h_1,\dots,h_ih_{i+1},\dots,h_k)\nonumber\\
-&(-1)^k\omega(z_j(g),h_1,\dots,h_{k-1}),\nonumber
\end{align}
the expression in \eqref{eq02BROOKS} reduces to 
\begin{align} 
\beta(g,h_1,\dots,h_k) = &\sum_{j=m+1}^{m+n-\ell+1}\chi_w(x_j \cdot \ldots \cdot x_{j+\ell-1})\omega(z_{j+\ell}(h_1), h_2, \dots, h_k)\nonumber\\
&-\sum_{j=1}^{m+n-\ell+1}\chi_w(x_j \cdot \ldots \cdot x_{j+\ell-1})\omega(z_{j+\ell}(gh_1), h_2, \dots, h_k)\nonumber\\
&+\sum_{j=1}^{m-\ell+1} \chi_w(x_j \cdot \ldots \cdot x_{j+\ell-1})\omega(z_j(g)h_1, h_2,\dots,h_{k}).\nonumber
\end{align}

Now the $j$-th terms in the first and second sum cancel for $j\geq m+1$ while the $j$-th terms in the second and third sum cancel for $j\leq m-\ell+1$, so that we are left with 
\begin{align} \label{eq05BROOKS}
\beta(g,h_1,\dots,h_k) = &-\sum_{j=m-\ell+2}^{m}\chi_w(x_j \cdot \ldots \cdot x_{j+\ell-1})\omega(z_{j+\ell}(gh_1), h_2, \dots, h_k)),\nonumber
\end{align}
which is bounded by $(\ell-1)\|\omega\|_\infty$. 
\end{proof}

\section{$\Delta$-decompositions and induced quasimorphisms of the free group}\label{section: decomposition}
We recall the notion of $\Delta$-decomposable quasimorphisms introduced by Heuer as a way of uniformizing the geometric properties of both Brooks quasimorphisms on non-selfoverlapping words and Rolli quasimorphims. For more details, the reader is referred to \cite[Section 3]{heuer17}

\subsection{$\Delta$-decompositions}
Let $F$ be a free group and $\mathcal{S}$ the symmetrization of a free generating set of $F$. For any symmetric set $\mathcal{P} \subset F$ satisfying $\id \not\in \mathcal{P}$ we call its elements \textit{pieces}. Furthermore, let $\mathcal{P}^*$ be the set of finite sequences of pieces, more precisely:
\[
\mathcal{P}^* := \{(g_1, \dots, g_K) \mid g_j \in \mathcal{P}, 1 \leq j \leq K < \infty \} \cup \{\text{empty sequence}\}.
\]
For any $g=(g_1, \dots, g_K)$ in $\mathcal{P}^*$, we shall write $|g|$ for the length $K$ of the tuple $(g_1,\dots,g_K)$, not to be confused with the length of $g$ with respect to $\s$.

\begin{dfn} \normalfont \label{def:Delta-dec}
A \textbf{$\Delta$-decomposition} of $F$ into the pieces $\mathcal{P}$ is a map 
\begin{align*}
\Delta: F &\longrightarrow \mathcal{P}^* \\
g &\longmapsto (g_1, \dots, g_K),
\end{align*}
that satisfies the following properties
\begin{enumerate} 
	\item[\textbf{[A]}] \textbf{for single words:}
	For every $g \in F$ with $\Delta(g)=(g_1, \dots, g_K)$ we have the following:
		\begin{itemize}
			\item $g=g_1 \cdot \ldots \cdot g_K$ as a reduced word,
			\item $\Delta(g^{-1})=(g_K^{-1}, \dots, g_1^{-1})$,
			\item $\Delta(g_i \cdot \ldots \cdot g_j)=(g_i, \dots, g_j)$ for all $1\leq i \leq j \leq K$.
		\end{itemize} 
	\item[\textbf{[B]}] \textbf{for products of two words:}
	There exists a constant $R>0$ such that for all $g,h \in F$, the sequences $\Delta(g)$, $\Delta(h)$ and $\Delta(gh)$ written in the following form
		\begin{align*}
		\Delta(g) &= \Delta(c_1^{-1})\Delta(r_1)\Delta(c_2),\\[0.4em]
		\Delta(h) &= \Delta(c_2^{-1})\Delta(r_2)\Delta(c_3) \text{ and} \\[0.4em]
		\Delta(gh)&= \Delta(c_1^{-1})\Delta(r_3^{-1})\Delta(c_3),
		\end{align*}
where $\Delta(c_i),\Delta(r_i) \in \mathcal{P}^*$ so that $\Delta(c_i)$ is of maximal length, satisfy
	\[
	\left|\Delta(r_i)\right| \leq R,
	\]
for $i=1,2,3$.
	\end{enumerate}
\end{dfn}
In order to refer to the system of equations in property \textbf{[B]} of a given $\Delta$-decomposition, we say for short that $g$ and $h$ form a \emph{$(g,h)$-triangle}. This reflects the geometric picture of property \textbf{[B]}, as illustrated in Figure \ref{triangle}.
		\begin{figure}[h]
			\center
			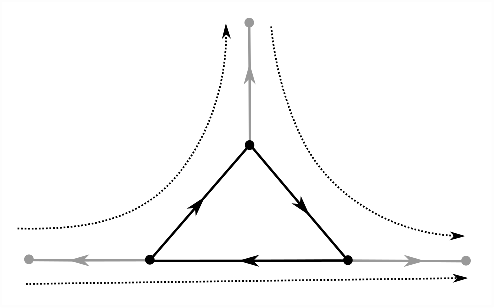
			\caption{}
			\label{triangle}
		\end{figure}
		
Moreover, we refer to the part of the $(g,h)$-triangle consisting of $c_1,c_2$ and $c_3$ as the \emph{$c$-part} and the part consisting of $r_1,r_2$ and $r_3$ as the (bounded) \emph{$r$-part}.
\subsection{$\Delta$-decomposable quasimorphisms}
Assigning a real value to each element of the pieces $\mathcal{P}$ in an appropriate way, ensures that any $\Delta$-decomposition of $F$ induces quasimorphisms on $F$. This is done in the following way:

\begin{dfn}\label{definition: delta quasimorphism}
Let $\Delta$ be a decomposition with pieces $\mathcal{P}$ and let $ \lambda \in \ell_{alt}^\infty(\mathcal{P})$, i.e. an alternating bounded map on $\mathcal{P}$. Then the map $\phi_{\lambda,\Delta}: F \to \R$ defined by
	\[
		 g \mapsto \sum_{j=1}^K \lambda(g_j),
	\]
where $\Delta(g) = (g_1, \dots, g_K)$, is called a \emph{$\Delta$-decomposable quasimorphism}.
\end{dfn}
This is indeed a legitimate name for this map which can be justified by using property \textbf{[B]} of the $\Delta$-decomposition:
\begin{prop}[Heuer, \cite{heuer17}]
Let $\Delta$ be a decomposition with pieces $\mathcal{P}$ and let $ \lambda \in \ell_{alt}^\infty(\mathcal{P})$.Then the map $\phi_{\lambda,\Delta}$ in Definition \ref{definition: delta quasimorphism} is a quasimorphism.
\end{prop}

By taking $\mathcal{P}:=\s$ and $\lambda$ arbitrary, one can recover all homomorphisms on the free group. 
We now give examples of $\Delta$-decomposable quasimorphisms: Brooks quasimorphisms on non-selfoverlapping words and Rolli quasimorphisms. 

\begin{exm}{\emph{$\Delta_w$-decomposition and Brooks quasimorphisms on non-selfoverlap\-ping words $w$.}}\;
Let $w \in F$ be a non-empty non-selfoverlapping word, meaning by definition that it cannot be written as $sms$, with $s$ non-empty. Note that in any reduced word $g\in F$, occurrences of $w$ and $w^{-1}$ as subwords cannot overlap. As $w$ is assumed to be non-selfoverlapping, any such $g$ can be represented uniquely as 
\[
g=u_1w^{\epsilon_1}u_2\cdot \ldots \cdot u_{k-1}w^{\epsilon_{k-1}}u_k,
\]
where each possibly empty $u_j$ is in reduced form, does not contain $w$ or $w^{-1}$ as subwords, and $\epsilon_j \in \{-1,+1\}$.
In this setting, define the pieces to be 
\[
\mathcal{P}_w = \{w,w^{-1}\} \cup \{u \in F\setminus \{ \mathrm{id}\}  \mid u \text{\ does not contain $w$ or $w^{-1}$ as subwords}\}.
\]

Then the \emph{$\Delta_w$-decomposition on the word $w$} is the map
\begin{align*}
\Delta_w: F &\to \mathcal{P}_w^*\\
g &\mapsto (u_1, w^{\epsilon_1}, u_2, \dots, u_{k-1},w^{\epsilon_{k-1}},u_k), 
\end{align*}
for $g$ as above.
The map $\lambda: \mathcal{P}_w \to \R$ given by 
\[
\lambda(p):= 
\left\{ \begin{array}{rl}
1, &\mathrm{if \ } p=w,\\
-1, &\mathrm{if \ } p=w^{-1},\\
0, &\mathrm{else},
\end{array}\right.
\]
induces the $\Delta_w$-decomposable quasimorphism $\phi_{\lambda, \Delta_{w}}$, namely precisely the Brooks quasimorphism on $w$. 
\end{exm}

\begin{exm}{\emph{$\Delta_{Rolli}$-decomposition and Rolli quasimorphisms.}} If $F$ is a free group of rank $\mathbf{n}$ and $\s=\{x_1,\dots,x_\mathbf{n}\}\cup \{x_1^{-1},\dots,x_\mathbf{n}^{-1}\}$ then 
any non-trivial $g \in F$ can be uniquely represented in the following form 
	\[
		g=x_{j_1}^{m_1}\cdot \ldots \cdot x_{j_k}^{m_k},
	\]
where all $m_j$ are non-zero and no consecutive $j_i\in \{1,\dots,\mathbf{n}\}$ are identical.
If we define the pieces to be $\mathcal{P}_{Rolli}=\{x_j^m \mid j \in \{1, \dots, \mathbf{n}\}, m \in \Z\setminus \{ 0\} \}$ then we define the \emph{$\Delta_{Rolli}$-decomposition} as the map 
	\begin{align*}
	\Delta_{Rolli}: F &\longrightarrow \mathcal{P}_{Rolli}^*\\
	g &\longmapsto (x_{j_1}^{m_1},\dots,x_{j_k}^{m_k}),
	\end{align*}
for $g$ as above.
Let $\lambda_1,\dots,\lambda_\mathbf{n} \in \ell_{alt}^\infty(\Z)$, namely bounded functions $\lambda_j: \Z \to \R$ satisfying $\lambda_j(-n)=-\lambda_j(n)$.
Then the map 
	\begin{align*}
	\lambda: \mathcal{P}_{Rolli} &\longrightarrow \R\\
	x_j^m &\longmapsto \lambda_j(m)
	\end{align*}
induces the quasimorphism $\phi_{\lambda,\Delta_{Rolli}}$, namely precisely the Rolli quasimorphism.
\end{exm}

\section{Proof of Theorem A (\normalfont{\lowercase{a}})}\label{section: proof}
Due to the discussion in Section \ref{section: aligned cochain complex} we can prove Theorem A (a) by showing that the cup product of the restriction to aligned cochains of cocycles representing the cohomology classes in consideration is a coboundary. Thus Theorem A (a) is a direct consequence of \begin{frthmAp*}\normalfont
Let $\phi$ be a $\Delta$-decomposable quasimorphism and for any $k >0$ let $\omega \in \mathscr{A}_b^k(F,\R)$. Then there exists $\beta  \in \mathscr{A}_b^{k+1}(F,\R)$ such that 
$$\boldsymbol{\delta} \beta = (\boldsymbol{\delta}\phi) \cup \omega.$$
\end{frthmAp*}

\begin{proof}Let $\phi$ and $\omega$ be as in Theorem A'  and fix a $\Delta$-decomposition on pieces $\p$ for which $\phi$ is $\Delta$-decomposable. Further let $\eta \in \mathscr{A}^k(F, \R)$ be defined by 
\begin{equation}\label{eq: choiceeta1}
\eta: (g,h_1,\dots,h_{k-1}) \longmapsto \sum_{j=1}^K \phi(g_{j})\omega(z_j(g),h_1,\dots,h_{k-1}),
\end{equation}
where $\Delta(g)=(g_1,\dots,g_K)$ and $z_j(g):=g_{j+1} \cdot \ldots \cdot g_{K}$ for $j=1,\dots,K$, and $\beta  \in \mathscr{A}_b^{k+1}(F,\R)$ by
$$ \beta=\phi \cup \omega +\boldsymbol{\delta} \eta.$$
By construction, $\boldsymbol{\delta} \beta = \boldsymbol{\delta}\phi \cup \omega$, so that we only need to show that $\beta$ is bounded. Let $(g,h_1,\dots,h_k)$ in $\mathcal{B}^{n}$ be arbitrary. Then by definition, we have that
\begin{align}\label{eq01}
\beta(g,h_1,\dots,h_k)
=&\phi(g)\omega(h_1,\dots,h_k) \\
+&\Big\{\eta(h_1,\dots,h_k)
-\eta(gh_1, h_2\dots,h_k)\nonumber\\
-&\sum_{i=1}^{k-1}(-1)^i\eta(g,h_1,\dots,h_ih_{i+1},\dots,h_k)\nonumber\\
-&(-1)^k\eta(g,h_1,\dots,h_{k-1})\Big\}\nonumber.
\end{align}
Let $\Delta(g)=(g_1,\dots,g_K), \Delta(h_1)=(h_{1,1},\dots,h_{1,L})$ and $\Delta(gh_1)=((gh_1)_1,\dots,(gh_1)_M)$ be the $\Delta$-decompositions of $g$, $h_1$ and $gh_1$. By expanding $\eta$ as in its definition \eqref{eq: choiceeta1} and by recalling that $\phi$ is as a $\Delta$-decomposable quasimorphism defined as (see Definition \ref{definition: delta quasimorphism}) \begin{equation*}
\phi(g)= \sum_{j=1}^K\phi(g_{j}),
\end{equation*}
we obtain 
\begin{align}\label{eq02}
&\quad\beta(g,h_1,\dots,h_k) \\[0.3em]&= 
\sum_{j=1}^K\phi(g_{j})
\omega(h_1,\dots,h_k)\nonumber\\
&+\sum_{j=1}^L\phi(h_{1,j})\omega(z_j(h_1), h_2, \dots, h_k)-\sum_{j=1}^M\phi((gh_1)_j)\omega(z_j(gh_1), h_2, \dots, h_k)\nonumber\\
&-\sum_{i=1}^{k-1}(-1)^i \sum_{j=1}^K\phi(g_{j})\omega(z_j(g), h_1,\dots,h_ih_{i+1},\dots,h_k)\nonumber\\
&-(-1)^k \sum_{j=1}^K \phi(g_{j})\omega(z_j(g),h_1,\dots,h_{k-1}).\nonumber
\end{align}
Using the cocycle relation 
\begin{align} \label{eq03}
\omega(z_j(g)h_1,\dots,h_k)
=&\omega(h_1,\dots,h_k)\\
-&\sum_{i=1}^{k-1}(-1)^i\omega(z_j(g),h_1,\dots,h_ih_{i+1},\dots,h_k)\nonumber\\
-&(-1)^k\omega(z_j(g),h_1,\dots,h_{k-1}),\nonumber
\end{align}
the expression in \eqref{eq02} reduces to 
\begin{align} \label{eq04}
\beta(g,h_1,\dots,h_k) = 
&\sum_{j=1}^K\phi(g_j)\omega(z_j(g)h_1,\dots,h_k)\\ \nonumber 
+&\sum_{j=1}^L\phi(h_{1,j})\omega(z_j(h_1),h_2,\dots,h_k)\\ \nonumber 
-&\sum_{j=1}^{M}\phi((gh_1)_j)\omega(z_j(gh_1),h_2,\dots,h_k).\nonumber
\end{align}

We now decompose the $(g,h_1)$-triangle into its $c$- and $r$-parts as in Definition \ref{def:Delta-dec}. Observe that since we are working with aligned chains, we have $c_2=\mathrm{id}$ so that
\begin{align*}
		\Delta(g) &= \Delta(c_1^{-1})\Delta(r_1),\\[0.4em]
		\Delta(h_1) &= \Delta(r_2)\Delta(c_3) \quad \text{and} \\[0.4em]
		\Delta(gh_1)&= \Delta(c_1^{-1})\Delta(r_3^{-1})\Delta(c_3).
		\end{align*}
We set $K'=|c_1|$ and $L'=|r_2|$ and note that since the lengths of the $r$-parts are bounded by $R$ we have 
$$|r_1|=K-K'\leq R, \quad |r_2|=L'\leq R \quad \mathrm{and} \quad |r_3|=M-K'-(L-L')\leq R.$$
Furthermore, by comparing the $c$-parts, we see that $\Delta(g)$ and $\Delta(gh_1)$ agree at the head, while $\Delta(h_1)$ and $\Delta(gh_1)$ agree at the tail of their corresponding $\Delta$-decompositions (see Figure \ref{triangle_picture1}). More precisely, we have 
\begin{align*}
g_j&=(gh_1)_j, \ \forall \ 1\leq j\leq K',\\
h_j&=(gh_1)_j,  \ \forall \ L'\leq j\leq L.
\end{align*}

\begin{figure}[h]
			\center
			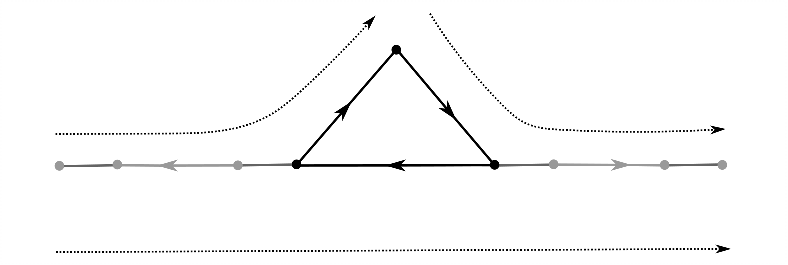
			\caption{}
			\label{triangle_picture1}
\end{figure}

Finally, if $1\leq j\leq K'$, then $z_j(g)h_1=z_j(gh_1)$, and if $L'\leq j\leq L$ then $z_j(h_1)=z_j(gh_1)$. This allows us to further reduce the expression for $\beta$ in \eqref{eq04} to 
 \begin{align*} 
\beta(g,h_1,\dots,h_k) = 
&\sum_{j=K'+1}^K\phi(g_j)\omega(z_j(g)h_1,\dots,h_k)\\[0.3em] 
+&\sum_{j=1}^{L'}\phi(h_{1,j})\omega(z_j(h_1),h_2,\dots,h_k)\\[0.3em] 
-&\sum_{j=K'+1}^{M-L+L'}\phi((gh_1)_j)\omega(z_j(gh_1),h_2,\dots,h_k),
\end{align*}
which is bounded since each of the three sums has at most $R$ summands, each of which is bounded since $\omega$ is bounded and $\phi$ is uniformly bounded on pieces of the $\Delta$-decomposition. 

\end{proof}

\begin{rem} Clearly the two proofs of Theorem A for (a) $\Delta$-decomposable quasimorphisms and (b) Brooks quasimorphisms have a similar flavour and could in fact be put in the following unified context: 
Given a $\Delta$-decomposition, $\ell>0$ and a bounded alternating function
$$\lambda:\mathcal{P}^\ell \longrightarrow \R,$$ 
by which we mean that $\lambda(g_1,\dots,g_\ell)=-\lambda(g_\ell^{-1},\dots,g_1^{-1})$, 
the function $\phi:F\rightarrow \R$ defined by
$$\phi(g)=\sum_{j=1}^{m-\ell+1} \lambda(g_j\cdot \ldots \cdot g_{j+\ell-1})$$
is a quasimorphism for which Theorem A holds. 
The case of $\Delta$-decomposable quasimorphisms corresponds to taking $\ell=1$, while Brooks quasimorphisms correspond to taking the trivial $\Delta$-decomposition with $\mathcal{P}=\mathcal{S}\cup \mathcal{S}^{-1}$ and 
$$\lambda(x_1,\dots,x_\ell)=\chi_w(x_1\cdot \ldots \cdot x_\ell)$$ 
for any word $w$ of length $\ell$. 

\end{rem}

\bibliographystyle{alpha}

\end{document}

%% file: triangle_picture2.eps_tex
\begingroup%
  \makeatletter%
  \providecommand\color[2][]{%
    \errmessage{(Inkscape) Color is used for the text in Inkscape, but the package 'color.sty' is not loaded}%
    \renewcommand\color[2][]{}%
  }%
  \providecommand\transparent[1]{%
    \errmessage{(Inkscape) Transparency is used (non-zero) for the text in Inkscape, but the package 'transparent.sty' is not loaded}%
    \renewcommand\transparent[1]{}%
  }%
  \providecommand\rotatebox[2]{#2}%
  \newcommand*\fsize{\dimexpr\f@size pt\relax}%
  \newcommand*\lineheight[1]{\fontsize{\fsize}{#1\fsize}\selectfont}%
  \ifx\svgwidth\undefined%
    \setlength{\unitlength}{292.57979796bp}%
    \ifx\svgscale\undefined%
      \relax%
    \else%
      \setlength{\unitlength}{\unitlength * \real{\svgscale}}%
    \fi%
  \else%
    \setlength{\unitlength}{\svgwidth}%
  \fi%
  \global\let\svgwidth\undefined%
  \global\let\svgscale\undefined%
  \makeatother%
  \begin{picture}(1,0.15152643)%
    \lineheight{1}%
    \setlength\tabcolsep{0pt}%
    \put(0,0){\includegraphics[width=\unitlength]{triangle_picture2.eps}}%
    \put(0.59611664,0.0611102){\color[rgb]{0,0,0}\makebox(0,0)[t]{\lineheight{1.25}\smash{\begin{tabular}[t]{c}$\cdots$\end{tabular}}}}%
    \put(0.04314837,0.01953102){\color[rgb]{0,0,0}\makebox(0,0)[t]{\lineheight{1.25}\smash{\begin{tabular}[t]{c}$\mathrm{id}$\end{tabular}}}}%
    \put(0.17367157,0.08248641){\color[rgb]{0.39215686,0.39215686,0.39215686}\makebox(0,0)[t]{\lineheight{1.25}\smash{\begin{tabular}[t]{c}$g_1$\end{tabular}}}}%
    \put(0.36395278,0.08411939){\color[rgb]{0.39215686,0.39215686,0.39215686}\makebox(0,0)[t]{\lineheight{1.25}\smash{\begin{tabular}[t]{c}$g_2$\end{tabular}}}}%
    \put(0.82072219,0.0846368){\color[rgb]{0.39215686,0.39215686,0.39215686}\makebox(0,0)[t]{\lineheight{1.25}\smash{\begin{tabular}[t]{c}$g_n$\end{tabular}}}}%
    \put(0.2656945,0.01983626){\color[rgb]{0,0,0}\makebox(0,0)[t]{\lineheight{1.25}\smash{\begin{tabular}[t]{c}$g_1$\end{tabular}}}}%
    \put(0.48563678,0.01964585){\color[rgb]{0,0,0}\makebox(0,0)[t]{\lineheight{1.25}\smash{\begin{tabular}[t]{c}$g_1g_2$\end{tabular}}}}%
    \put(0.72127443,0.02040169){\color[rgb]{0,0,0}\makebox(0,0)[t]{\lineheight{1.25}\smash{\begin{tabular}[t]{c}$g_1\cdot \ldots \cdot g_{n-1}$\end{tabular}}}}%
    \put(0.9386168,0.02240562){\color[rgb]{0,0,0}\makebox(0,0)[t]{\lineheight{1.25}\smash{\begin{tabular}[t]{c}$g_1\cdot \ldots \cdot g_n$\end{tabular}}}}%
  \end{picture}%
\endgroup%

%% file: triangle.eps_tex
\begingroup%
  \makeatletter%
  \providecommand\color[2][]{%
    \errmessage{(Inkscape) Color is used for the text in Inkscape, but the package 'color.sty' is not loaded}%
    \renewcommand\color[2][]{}%
  }%
  \providecommand\transparent[1]{%
    \errmessage{(Inkscape) Transparency is used (non-zero) for the text in Inkscape, but the package 'transparent.sty' is not loaded}%
    \renewcommand\transparent[1]{}%
  }%
  \providecommand\rotatebox[2]{#2}%
  \newcommand*\fsize{\dimexpr\f@size pt\relax}%
  \newcommand*\lineheight[1]{\fontsize{\fsize}{#1\fsize}\selectfont}%
  \ifx\svgwidth\undefined%
    \setlength{\unitlength}{237.05727454bp}%
    \ifx\svgscale\undefined%
      \relax%
    \else%
      \setlength{\unitlength}{\unitlength * \real{\svgscale}}%
    \fi%
  \else%
    \setlength{\unitlength}{\svgwidth}%
  \fi%
  \global\let\svgwidth\undefined%
  \global\let\svgscale\undefined%
  \makeatother%
  \begin{picture}(1,0.62283179)%
    \lineheight{1}%
    \setlength\tabcolsep{0pt}%
    \put(0,0){\includegraphics[width=\unitlength]{triangle.eps}}%
    \put(0.19022991,0.11714986){\color[rgb]{0.6,0.6,0.6}\makebox(0,0)[lt]{\lineheight{1.25}\smash{\begin{tabular}[t]{l}$c_1$\end{tabular}}}}%
    \put(0.41528266,0.1882647){\color[rgb]{0,0,0}\makebox(0,0)[lt]{\lineheight{1.25}\smash{\begin{tabular}[t]{l}$r_1$\end{tabular}}}}%
    \put(0.55932987,0.18414565){\color[rgb]{0,0,0}\makebox(0,0)[lt]{\lineheight{1.25}\smash{\begin{tabular}[t]{l}$r_2$\end{tabular}}}}%
    \put(0.48207706,0.12552628){\color[rgb]{0,0,0}\makebox(0,0)[lt]{\lineheight{1.25}\smash{\begin{tabular}[t]{l}$r_3$\end{tabular}}}}%
    \put(0.31379668,0.28925473){\color[rgb]{0,0,0}\makebox(0,0)[lt]{\lineheight{1.25}\smash{\begin{tabular}[t]{l}$g$\end{tabular}}}}%
    \put(0.68059131,0.2792938){\color[rgb]{0,0,0}\makebox(0,0)[lt]{\lineheight{1.25}\smash{\begin{tabular}[t]{l}$h$\end{tabular}}}}%
    \put(0.46850816,0.01327039){\color[rgb]{0,0,0}\makebox(0,0)[lt]{\lineheight{1.25}\smash{\begin{tabular}[t]{l}$gh$\end{tabular}}}}%
    \put(0.76381926,0.11748974){\color[rgb]{0.6,0.6,0.6}\makebox(0,0)[lt]{\lineheight{1.25}\smash{\begin{tabular}[t]{l}$c_3$\end{tabular}}}}%
    \put(0.51099169,0.42240933){\color[rgb]{0.6,0.6,0.6}\makebox(0,0)[lt]{\lineheight{1.25}\smash{\begin{tabular}[t]{l}$c_2$\end{tabular}}}}%
  \end{picture}%
\endgroup%

%% file: triangle_picture1.eps_tex
\begingroup%
  \makeatletter%
  \providecommand\color[2][]{%
    \errmessage{(Inkscape) Color is used for the text in Inkscape, but the package 'color.sty' is not loaded}%
    \renewcommand\color[2][]{}%
  }%
  \providecommand\transparent[1]{%
    \errmessage{(Inkscape) Transparency is used (non-zero) for the text in Inkscape, but the package 'transparent.sty' is not loaded}%
    \renewcommand\transparent[1]{}%
  }%
  \providecommand\rotatebox[2]{#2}%
  \newcommand*\fsize{\dimexpr\f@size pt\relax}%
  \newcommand*\lineheight[1]{\fontsize{\fsize}{#1\fsize}\selectfont}%
  \ifx\svgwidth\undefined%
    \setlength{\unitlength}{380.96177826bp}%
    \ifx\svgscale\undefined%
      \relax%
    \else%
      \setlength{\unitlength}{\unitlength * \real{\svgscale}}%
    \fi%
  \else%
    \setlength{\unitlength}{\svgwidth}%
  \fi%
  \global\let\svgwidth\undefined%
  \global\let\svgscale\undefined%
  \makeatother%
  \begin{picture}(1,0.33255308)%
    \lineheight{1}%
    \setlength\tabcolsep{0pt}%
    \put(0,0){\includegraphics[width=\unitlength]{triangle_picture1.eps}}%
    \put(0.22261798,0.13993792){\color[rgb]{0.6,0.6,0.6}\makebox(0,0)[lt]{\lineheight{1.25}\smash{\begin{tabular}[t]{l}$c_1$\end{tabular}}}}%
    \put(0.4431924,0.18287519){\color[rgb]{0,0,0}\makebox(0,0)[lt]{\lineheight{1.25}\smash{\begin{tabular}[t]{l}$r_1$\end{tabular}}}}%
    \put(0.53282716,0.18031204){\color[rgb]{0,0,0}\makebox(0,0)[lt]{\lineheight{1.25}\smash{\begin{tabular}[t]{l}$r_2$\end{tabular}}}}%
    \put(0.48475587,0.14383559){\color[rgb]{0,0,0}\makebox(0,0)[lt]{\lineheight{1.25}\smash{\begin{tabular}[t]{l}$r_3$\end{tabular}}}}%
    \put(0.31955968,0.21021684){\color[rgb]{0,0,0}\makebox(0,0)[lt]{\lineheight{1.25}\smash{\begin{tabular}[t]{l}$g$\end{tabular}}}}%
    \put(0.66613588,0.20467596){\color[rgb]{0,0,0}\makebox(0,0)[lt]{\lineheight{1.25}\smash{\begin{tabular}[t]{l}$h_1$\end{tabular}}}}%
    \put(0.48384941,0.03076236){\color[rgb]{0,0,0}\makebox(0,0)[lt]{\lineheight{1.25}\smash{\begin{tabular}[t]{l}$gh_1$\end{tabular}}}}%
    \put(0.73633275,0.14014953){\color[rgb]{0.6,0.6,0.6}\makebox(0,0)[lt]{\lineheight{1.25}\smash{\begin{tabular}[t]{l}$c_3$\end{tabular}}}}%
    \put(0.22971473,0.08848297){\color[rgb]{0.6,0.6,0.6}\makebox(0,0)[t]{\lineheight{1.25}\smash{\begin{tabular}[t]{c}$\cdots$\end{tabular}}}}%
    \put(0.09921283,0.10228312){\color[rgb]{0.39215686,0.39215686,0.39215686}\makebox(0,0)[lt]{\lineheight{1.25}\smash{\begin{tabular}[t]{l}$g_1$\end{tabular}}}}%
    \put(0.31994608,0.09810684){\color[rgb]{0.39215686,0.39215686,0.39215686}\makebox(0,0)[lt]{\lineheight{1.25}\smash{\begin{tabular}[t]{l}$g_{K'}$\end{tabular}}}}%
    \put(0.64125534,0.10053433){\color[rgb]{0.47058824,0.47058824,0.47058824}\makebox(0,0)[lt]{\lineheight{1.25}\smash{\begin{tabular}[t]{l}$h_{1,L'+1}$\end{tabular}}}}%
    \put(0.85710667,0.09885059){\color[rgb]{0.47058824,0.47058824,0.47058824}\makebox(0,0)[lt]{\lineheight{1.25}\smash{\begin{tabular}[t]{l}$h_{1,L}$\end{tabular}}}}%
    \put(0.79759721,0.08796552){\color[rgb]{0.6,0.6,0.6}\makebox(0,0)[t]{\lineheight{1.25}\smash{\begin{tabular}[t]{c}$\cdots$\end{tabular}}}}%
    \put(0.08756599,0.0416652){\color[rgb]{0.39215686,0.39215686,0.39215686}\makebox(0,0)[lt]{\lineheight{1.25}\smash{\begin{tabular}[t]{l}$(gh_1)_1$\end{tabular}}}}%
    \put(0.10741452,0.07905695){\color[rgb]{0.39215686,0.39215686,0.39215686}\rotatebox{-88.570757}{\makebox(0,0)[t]{\lineheight{1.25}\smash{\begin{tabular}[t]{c}$=$\end{tabular}}}}}%
    \put(0.31339197,0.04275505){\color[rgb]{0.39215686,0.39215686,0.39215686}\makebox(0,0)[lt]{\lineheight{1.25}\smash{\begin{tabular}[t]{l}$(gh_1)_{K'}$\end{tabular}}}}%
    \put(0.33179779,0.07568234){\color[rgb]{0.39215686,0.39215686,0.39215686}\rotatebox{-88.570757}{\makebox(0,0)[t]{\lineheight{1.25}\smash{\begin{tabular}[t]{c}$=$\end{tabular}}}}}%
    \put(0.63248872,0.04192403){\color[rgb]{0.47058824,0.47058824,0.47058824}\makebox(0,0)[lt]{\lineheight{1.25}\smash{\begin{tabular}[t]{l}$(gh_1)_{L'+1}$\end{tabular}}}}%
    \put(0.6516805,0.07673758){\color[rgb]{0.47058824,0.47058824,0.47058824}\rotatebox{-88.570757}{\makebox(0,0)[t]{\lineheight{1.25}\smash{\begin{tabular}[t]{c}$=$\end{tabular}}}}}%
    \put(0.84728089,0.04402319){\color[rgb]{0.47058824,0.47058824,0.47058824}\makebox(0,0)[lt]{\lineheight{1.25}\smash{\begin{tabular}[t]{l}$(gh_1)_L$\end{tabular}}}}%
    \put(0.86542685,0.0765406){\color[rgb]{0.47058824,0.47058824,0.47058824}\rotatebox{-88.570757}{\makebox(0,0)[t]{\lineheight{1.25}\smash{\begin{tabular}[t]{c}$=$\end{tabular}}}}}%
  \end{picture}%
\endgroup%